\RequirePackage{fix-cm}
\documentclass[smallextended]{svjour3}       
\smartqed  
\usepackage{graphicx}
\usepackage{amssymb,subeqnarray}
\usepackage{epsfig}
\usepackage{subfigure}
\usepackage{color}

\usepackage{amsmath, amssymb, latexsym, amscd}
\usepackage{changepage}
\usepackage{pstricks,pst-node,pst-grad,pst-3d,hhline,multirow,graphicx,caption}

\newcommand{\Reals}[1]{{\rm I\! R}^{#1}}

\newcommand{\jump}[1]{\lbrack\!\lbrack #1 \rbrack\!\rbrack} 
\newcommand{\av}[1]{\{\!\!\{#1\}\!\!\}}
\newtheorem{algorithm}{Algorithm}[section]

\newcommand{\norm}[1]{|\!|#1|\!|}
\newcommand{\tnorm}[1]{|\!|\!|#1|\!|\!|}

\newcommand{\what}[1]{\widehat{#1}}
\newcommand{\wtilde}{\widetilde}

\def\T{\mathcal T}

\def\begp{\begin{pmatrix}}
\def\ep{\end{pmatrix}}
\def\begm{\begin{matrix}}
\def\edma{\end{matrix}}

\def\b0{{\mathbf 0}}

\def\bn{\mathbf{n}}

\def\bt{{\mathbf t}}

\def\bu{\mathbf{u}}

\def\bw{\mathbf{w}}

\def\bz{\mathbf{z}}
 \newsavebox{\savepar}

\numberwithin{equation}{section}

\newcommand{\tend}[1]{\hbox{\oalign{${#1}$\crcr\hidewidth$\scriptscriptstyle\bf{\sim}$\hidewidth}}}
\newcommand{\tenq}[1]{\hbox{\oalign{${#1}$\crcr\hidewidth$\scriptscriptstyle\bf{\approx}$\hidewidth}}}

\begin{document}

\title{Locally conservative immersed finite element method for elliptic interface problems
}

\titlerunning{Locally conservative IFEM for elliptic interface problems}        

\author{Gwanghyun Jo   \and   Do Y. Kwak \and Young-Ju Lee}

\authorrunning{Gwanghyun Jo, Do Y. Kwak and Young-Ju Lee}

\institute{G. Jo \at
              Department of Mathematics, Kunsan National University, Republic of Korea \\
              \email{gwanghyun@kunsan.ac.kr}           
           \and
           D. Y. Kwak \at
              Department of Mathematical Sciences, KAIST, Daejeon, Republic of Korea \\
            \email{kdy@kaist.ac.kr}
             \and
           Y.-J. Lee \at
             Department of Mathematics, Texas State University, U.S.A \\
            \email{yjlee@txstate.edu}
}
\date{Received: date / Accepted: date}
\maketitle

\begin{abstract}
In this paper, we introduce the locally conservative enriched immersed finite element method (EIFEM) to tackle the elliptic problem with interface. The immersed finite element is useful for handling interface with mesh unfit with the interface. However, all the currently available method under IFEM framework may not be designed to consider the flux conservation. We provide an efficient and effective remedy for this issue by introducing a local piecewise constant enrichment, which provides the locally conservative flux. We have also constructed and analyzed an auxiliary space preconditioner for the resulting system based on the application of algebraic multigrid method. The new observation in this work is that by imposing strong Dirichlet boundary condition for the standard IFEM part of EIFEM, we are able to remove the zero eigen-mode of the EIFEM system while still imposing the Dirichlet boundary condition weakly assigned to the piecewise constant enrichment part of EIFEM. A couple of issues relevant to the piecewise constant enrichment given for the mesh unfit to the interface has been discussed and clarified as well. Numerical tests are provided to confirm the theoretical development.
\end{abstract}

\section{Introduction}
There are many problems in engineering areas whose governing 
equations are described by a combined system of elliptic equations and transport equations. When solving these problems numerically, accurately predicting flow variables is as important as estimating displacements. One of the criteria for assessing the stability of the numerically resolved flow is whether it has acquired a local conservation. Without a local conservation in flow, the transport variable may suffer a nonphysical result if there is a spurious source.

Various locally conservative schemes were developed in finite element method (FEM) community, which include mixed finite element methods (MFEMs) \cite{raviart1977mixed,brezzi1991mixed}, CG flux \cite{chippada1998projection,hughes2000continuous,larson2004conservative,cockburn2007locally}, and some discontinuous Galerkin (DG) methods combined with post-processing technique for resolving flows \cite{Bastian2003Superconvergence,Ern2007accurate}.
Recently, a conservative method called enriched Galerkin (EG), similar to the DG, but which has a much less DOF than that of DG, is introduced  \cite{sun2009locally,lee2015locally}. EG enriches the conforming finite element space with a piecewise constant. This can produce a locally conservative flux effectively. 

In the perspective of solving the discretized system, the data structure becomes complicated if the nature of the medium underlying the governing equation becomes discontinuous along some interfaces. This is because if there is an interface, one has to use a fitted grid whose nodes are aligned on the interface.
Thus, one may ask if we can devise a conservative scheme which is more efficient when solving a problem with an interface.
Recently, various structured grids based methods were developed, for example, extended finite element methods \cite{moes1999finite,belytschko1999elastic,krysl2000efficient,belytschko2003structured,legrain2005stress}, immersed finite element method (IFEM) \cite{li2003new,li2004immersed,chou2010optimal,kwak2010analysis,Lin2015partially,kwak2017stabilized,jo2019recent}, etc.. See also an interesting contribution by Guzman et al. for elliptic problems with interface with higher-order finite element methods \cite{guzman2016higher}.  
Among many available methods, we consider to use IFEM, which uses a strategy of modifying the basis along the interfaces. IFEM has the advantage that an extra degree of freedom is not required, and thus it can be applied effectively for various equations,  for examples elliptic equations \cite{chou2010optimal,kwak2010analysis,Lin2015partially}, two-phase flows in the porous media \cite{jo2017impes}, elasticity equation \cite{kwak2017stabilized,kyeong2017immersed,jo2020stabilized}, and Poisson Bolzamann equation \cite{Kwon2018Discontinuous}.
In addition, because of its simple data structure, geometric multigrid algorithms have been efficiently applied to solve the discretized system resulting from IFEM \cite{jo2017impes,jo2018geometric}, while the performance of algebraic version multigrid was reported in \cite{feng2014immersed}.

In this work, we propose a novel methodology to compute flows through a nonhomogeneous media using IFEM.
To use a structured grid, the $P_1$-conforming basis functions are modified so that the flux continuity conditions are satisfied.
Next, to keep the  mass conservation, the modified space is enriched by piecewise constant functions.
Since the resulting space is discontinuous across the edge, the bilinear form used to solve the elliptic equation contains a term that compensates the difference in the normal flux from the two adjacent elements along the edges.
After the equation is solved for the displacement variable, the flow variable can be obtained  locally on each edge, which is a  similar technique used in EG.
We name our method an enriched immersed finite element method (EIFEM).
Also, we have developed and analyzed an auxiliary space preconditioner based on algebraic multigrid method for solving the algebraic system arising from EIFEM.

The novelties in this work are that:
1) Both displacement and flow variables can be approximated on a structured grid, regardless of interface,
2) The data structure is simple, thus the effective solver based on subspace correction method can be applied easily,
3) The  pressure variable is obtained by solving a symmetric problem by (preconditioned) conjugate gradient, while the flux variable is computed locally, thus the whole implementation is simpler than MFEM and CG-flux.

The rest of the paper is organized as follows.
In Section 2, we write the model problem and review the IFEM space.
EIFEM is proposed in Section 3 the analysis of it is provided in Section 4.
In Section 5, we present and analyze an auxiliary space preconditioner based on algebraic multigrid method. The numerical results are given in Section 6. Lastly, we offer concluding remarks in Section 7.

Throughout the paper, we shall set $C, C_t, \tilde{C}$ will denote generic positive constants independent of the mesh size $h$ or functions involved, not necessarily the same for each appearance. Oftentimes, we shall use the following notation:
$$
\exists \mbox{ a generic constant } C > 0 \mbox{ such that } A < C B \quad \Leftrightarrow \quad A \lesssim B,
$$
and
$$
\exists \mbox{ a generic constant } C > 0 \mbox{ such that } A > C B \quad \Leftrightarrow \quad A \gtrsim B.
$$

\section{Governing Equations}

In this section, we shall introduce a couple of useful notation and present our governing equations of interest.

We assume that $\Omega$ is a convex polygonal domain in $\Reals{2}$ and it is decomposed into the following form:
\begin{equation}
\Omega = \Omega_1 \cup \Gamma \cup \Omega_2,
\end{equation}
where $\Omega_1$ and $\Omega_2$ are subdomains of $\Omega$ with different elastic materials having distinct Lam\'{e} constants, and $\Gamma$ is the interface between these domains. For any bounded subdomain $D \subset \Omega$, its restriction onto $\Omega_1$ and $\Omega_2$, are denoted by $D_1$ and $D_2$, respectively, i.e., $D_1 := D \cap \Omega_1$ and $D_2 := D \cap \Omega_2$.

We shall assume that $\Gamma$ is $C^2$ interface. Under this setting, the boundary of $\Omega$, denoted by $\partial \Omega$ is given as follows:
\begin{equation}
\partial \Omega = \partial \Omega_1 \cup \partial \Gamma \cup \partial \Omega_2,
\end{equation}
where $\partial \Omega_1, \partial \Gamma$ and $\partial \Omega_2$ are boundaries of $\Omega_1, \Gamma$ and $\Omega_2$, respectively. Let $v$ be a function defined on $\Omega$.  

We shall use standard function spaces. For a given subdomain $D \subset \Omega$, $C^m(D)$ denotes the space of the first m-derivatives are continuous in $D$, $H^m(D)$, $H^1_0(D)$, $H^m(\partial D)$ are the ordinary Sobolev spaces of order $m$ with the norm $\norm{\cdot}_{m,D}$ and the semi-norm $|\cdot|_{m,D}$. For $m = 0$, $(\cdot,\cdot)_{0,D}$ denote a $L^2$-inner product on the domain $D$. In case $D = \Omega$, the norm $\norm{\cdot}_{m,\Omega}$ and the inner product $(\cdot,\cdot)_{m,\Omega}$ shall be denoted simply by $\norm{\cdot}_{m}$ and $(\cdot,\cdot)_{m,\Omega}$, respectively. For $m = 1,2$, we also introduce the broken Sobolev space $\wtilde{H}^m(D)$ defined as
\begin{eqnarray*}
\wtilde{H}^m (D) :=\{u\in H^{m-1}(D) ,\ | \, u|_{D_1} \in H^m(D_1) \mbox{ and } u|_{D_2} \in H^m(D_2) \},
\end{eqnarray*}
equipped with the norm:
\begin{eqnarray*}
\norm{u}^2_{\wtilde{m},D} := \norm{u}_{m-1, D}^2+ \norm{u}^2_{m, D_1} + \norm{u}^2_{m, D_2}.
\end{eqnarray*}

Let $\beta$ be the conductivity for a given domain $\Omega$, i.e., the ratio between the permeability and viscosity, which will be allowed to be discontinuous across the interface $\Gamma$. We shall assume that $\beta$ is bounded and uniformly positive in $\Omega$ with $\beta|_{\Omega_1} \in C^1(\Omega_1)$ and $\beta|_{\Omega_2} \in C^1(\Omega_2)$. We let $\underline{\beta}$ and $\overline{\beta}$ be the lower and upper bound of $\beta$, respectively, i.e.,
\begin{equation}
0 < \underline{\beta} < \beta < \overline{\beta}.
\end{equation}
We now introduce two additional Sobolev spaces for taking into account the boundary and interface conditions.Namely,
\begin{subeqnarray}
\wtilde{H}^1_0 (\Omega) &:=& \{u \in \wtilde{H}^1 (\Omega) \, | \, u=0 \, \, {\rm on} \, \, \partial \Omega \},  \\
\wtilde{H}^2_{\Gamma_\beta}(\Omega) &:=& \{v \in \widetilde{H}^{2}(\Omega)\, | \, \jump{\beta \nabla v} = 0 \mbox{ on } \Gamma\}.
\end{subeqnarray}

The second order elliptic interface model problem that we are aiming to solve is that, given ${\bf{f}} \in L^2(\Omega)$, find $p \in \widetilde{H}^1_0(\Omega)$ such that
\begin{subeqnarray}\label{governing}
{\rm div} \bu &=& \mathbf{f}, \quad \mbox{in } \Omega,     \label{governing_eq1} \\
\bu &=& -\beta \nabla p,  \quad \mbox{in } \Omega,           \label{governing_eq2} \\
\jump{p} &=& 0, \quad \mbox{ on } \Gamma,                     \label{governing_eq3} \\
\jump{\beta \nabla p} &=& 0, \quad \mbox{ on } \Gamma, \label{governing_eq4} \\
p &=& 0, \quad \mbox{ in } \partial \Omega,                      \label{governing_eq5}
\end{subeqnarray}
where $\jump{p}$ and $\jump{\beta\nabla p}$ denote the jump of the function $p$ and the jump of $(\beta \nabla p) \cdot \bn_\Gamma$ on $\Gamma$, respectively. Here $\bn_\Gamma$ is the normal to the interface $\Gamma$ (further discussion on this notation will be introduced below). 

We note that for the sake of simplicity, the problem \eqref{governing} imposes homogeneous boundary and interface conditions. However, non-homogeneous conditions can also be considered with a simple modification. The weak formulation of the model problem \eqref{governing_eq1} is given as follows: find $p \in \widetilde{H}^1_0(\Omega)$ such that
\begin{equation}\label{con_weak_form}
\int_{\Omega} \beta  \nabla p \cdot \nabla v\, {\rm d} \mathbf{x} =\int_{\Omega} f v\, {\rm d} \mathbf{x}
\end{equation}
for all $v\in H^1_0(\Omega)$. Note that while $\jump{p} = 0$ on $\Gamma$ and $p = 0$ on $\partial \Omega$ are essential conditions, $\jump{\beta \nabla p} = 0$ on $\Gamma$ is the natural boundary condition.

Finally, we state the following regularity theorem \cite{bramble1996finite,chen1998finite,rouitberg1969theorem} regarding the model problem \eqref{con_weak_form}.
\begin{proposition}\label{main:them}
Let $f \in L^2(\Omega)$. Then, there exist a unique solution $p \in \widetilde{H}^1_0(\Omega)$ of problem \eqref{governing_eq1} such that
\begin{equation}\label{C_R}
\norm{p}_{\wtilde{H}^2(\Omega)} \lesssim \norm{f}_{L^2(\Omega)}.
\end{equation}
\end{proposition}

\subsection{Immersed finite element method for \eqref{con_weak_form}}

In this section, we review and discuss the classical immersed finite element method to handle the problem \eqref{con_weak_form}. Let $\mathcal{T}_h$ be a regular triangulation of $\Omega$. We note that the triangulation is provided in general for which nodes are not necessarily aligned with the interface $\Gamma$. Under this setting, there are two types of triangles in $\mathcal{T}_h$, i.e., an interface element $T$, which is characterized by the fact that it is cut by the interface $\Gamma$ and a non-interface element $T$ which is not. We shall denote $\mathcal{T}_{h,\Gamma} \subset \mathcal{T}_h$, by the set of interface triangles. We shall let $\mathcal{E}_h$ be the set of edges of $\mathcal{T}_h$. Note that $\mathcal{E}_h = \mathcal{E}_h^o \cup \mathcal{E}_h^\partial$, where $\mathcal{E}_h^o$ is the set of interior edges while $\mathcal{E}_h^\partial$ is the set of boundary edges.

The space $H^s(\mathcal{T}_h) (s \in \Reals{})$ is the set of element-wise $H^2$ functions on $\mathcal{T}_h$, and $L^2(\mathcal{T}_h)$ refers to the set of element-wise $L^2$ functions. Following \cite{arnold2002unified}, for any $e \in \mathcal{E}_h$, we denote by $|e|$ the length of the edges $e$. Now let $e \in \mathcal{E}_h^o$, $T^{+}$ and $T^{-}$ denote two neighboring elements such that  $e = \partial T^{+}\cap \partial T^{-}$. Let $\bn^{+}$ and $\bn^{-}$ be the outward normal unit vectors to $\partial T^+$ and $\partial T^-$, respectively. For any given function $\xi$ and vector function $\tend{\xi}$, defined on the triangulation $\mathcal{T}_h$, we denote $\xi^{\pm}$ and ${\bf{\xi}}^{\pm}$ by the restrictions of $\xi$ and ${\bf{\xi}}$ to $T^\pm$, respectively. We define the average $\av{\cdot}$ as follows: for $\zeta \in L^2(\mathcal{T}_h)$ and $\tend{\tau} \in L^2(\mathcal{T}_h)^d$,
\begin{equation}\label{av-w}
\av{\zeta} := \frac{1}{2} \left ( \zeta^+ + \zeta^- \right ) \quad \mbox{ and } \quad \av{\tend{\tau}} := \frac{1}{2} \left (\tend{\tau}^+ + \tend{\tau}^- \right ) \quad \mbox{on } e \in \mathcal{E}_h^o.  
\end{equation}
On the other hand, for $e \in \mathcal{E}_h^\partial$, we set $\av{\zeta} := \zeta$ and $\av{\tend{\tau}} := \tend{\tau}$. The jump across the interior edge will be defined as usual: 
\begin{equation}
\jump{\zeta} = \zeta^+\bn^+ + \zeta^-\bn^- \quad \mbox{ and } \quad \jump{\tend{\tau}} = \tend{\tau}^+\cdot \bn^+ + \tend{\tau}^-\cdot\bn^- \quad \mbox{on } e\in \mathcal{E}_h^o. 
\end{equation}
For $e\in \mathcal{E}_h^\partial$, we set $\jump{\zeta} = \zeta \bn$. For any given edge $e \in \mathcal{E}_h^o$, there are two choices of the normal $\bn$ to $e$ and it is useful to fix one of them, for example, for an appropriate definition of the flux. Such a fixed choice of the normal to $e$ will be denoted by $\bn_e.$   

For any $T \in \mathcal{T}_h$ and an inner product $(\cdot,\cdot)_{m,T}$, the computation can be done with the following decomposition:
\begin{equation}
(u,v)_{m,T} = (u,v)_{m,T^+} + (u,v)_{m,T^-}.
\end{equation}

We will let $\mathcal{P}^k(T)$ denote the space of polynomials of degrees less than or equal to $k$ for a given $T \in \mathcal{T}_h$. Similarly, we also let $\mathcal{P}^k(e)$ denote the space of polynomials of degree less than or equal to $k$ for a given $e \in \mathcal{E}_h$.
\begin{figure}[ht]
\centering
\includegraphics[height = 0.3\textwidth]{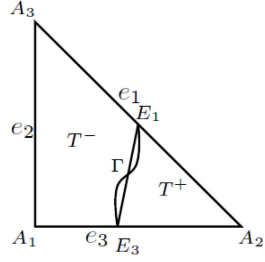}
\caption{An interface element $T$ cut by interface $\Gamma$.}
\end{figure}\label{fig:interel}
Given a non-interface element $T$, we shall recall that $S_h(T) = \mathcal{P}^1(T) = {\rm span} \{\lambda_j\}_{j=1,2,3}$, where $\lambda_j$'s are standard barycentric coordinates, i.e., $\lambda_j(A_i) = \delta_{ij}$, where $A_i$'s are nodes of $T$ and $\delta_{ij}$ is the standard Kronecker delta function. In case $T$ is an interface element, the space $S_h(T)$ will be modified as $\widehat{S}_h(T) = {\rm span} \{ \widehat{\lambda}_j \}_{j=1,2,3}$, so that $\widehat{\lambda}_j$ satisfies the interface condition as well as $\widehat{\lambda}_j(A_i) = \delta_{ij}$. More precisely, as described in Figure \ref{fig:interel}, given an interface element $T$, we suppose the interface $\Gamma$ cuts $T$ through edges $e_1$ and $e_3$ of $T$ at points $E_2$ and $E_1$. This cut will divide $T$ into  $T^+$ contained in $\Omega^+$ and $T^-$ contained in $\Omega^-$.

For $j=1,2,3$, we modify $\lambda_j \in S_h(T)$ to a piecewise linear function $\widehat{\lambda}_j$ of the following form:
\begin{align} \label{ifem_basis}
\what{\lambda}_j(X):=\left\{\begin{array}{c}
\what{\lambda}^+_j(X) = a_j^+ + b_j^+x+ c_j^+y,  \quad X = (x,y) \in T^+, \\
\what{\lambda}^-_j(X) = a_j^- + b_j^-x+ c_j^-y,  \quad X = (x,y) \in T^-,
\end{array}
\right.
\end{align}
where the coefficients $\{a_j^\pm, b_j^\pm, c_j^\pm \}$ are determined by the jump conditions and vertex degrees of freedom:
\begin{eqnarray}\label{intcond}
\what{\lambda}_j(X_i) = \delta_{ij}, \mbox{ and } \quad \jump{\what{\lambda}_j} = \jump{\beta \nabla \what{\lambda}_j} = 0 \mbox{ on } \Gamma.
\end{eqnarray}
It is well known that the aforementioned conditions (\ref{intcond}) uniquely determine $\what{\lambda}_j$ for $j=1,2,3$, (see \cite{chou2010optimal}). The vertex based piecewise linear immersed finite element space (IFEM) for the problem \eqref{con_weak_form} is then characterized by the space $\widehat{S}_h(\Omega)$ defined as follows: for $\phi \in \widehat{S}_h(\Omega)$, it holds
\begin{eqnarray}
\phi|_T \in S_h(T), && \hbox{if } T \in \mathcal{T}_h /\ \mathcal{T}_{h,\Gamma}, \\
\phi|_T \in \widehat{S}_h(T),  && \hbox{if }  T \in \mathcal{T}_{h,\Gamma}, \\
\phi|_{T_1}(X) = \phi|_{T_2}(X), && \hbox{if } X \in  T_1 \cap T_2, \\
\phi(X) = 0, && \hbox{if } X \in \partial \Omega.
\end{eqnarray}
The approximation property of the space $\widehat{S}_h(\Omega)$ is well known, \cite{he2008approximation,kwak2010analysis,li2004immersed}. Namely, let $\ell_h : H^2(T) \to \widehat{S}_h(T)$ be the standard local nodal interpolation operator defined by, for $v \in H^2(T)$,
\begin{equation}
\ell_h (v)(A_i) = v(A_i) \quad i = 1,2,3.
\end{equation}
We then let $I_h : \wtilde{H}^2(\Omega) \to \widehat{S}_h(\Omega)$ be the extension of $\ell_h$ defined by $I_h(v)|_T = \ell_h(v|_T)$. The following has been established:
\begin{lemma}\label{C_pi}
There exists a constant $C>0$ such that
\begin{align*}
\sum_{T \in \mathcal{T}_h} (\norm{\phi - I_h \phi}_{0,T}+h\norm{\phi - I_h \phi}_{1,T}) \leq C h^2 \norm{\phi}_{\wtilde{H}^2(\Omega)},
\end{align*}
for all $\phi \in \widetilde{H}^2_{\Gamma_\beta}(\Omega)$.
\end{lemma}
The classic IFEM construct a finite element solution to the problem \eqref{con_weak_form}, by solving the following discrete weak formulation: find $p_h \in \widehat{S}_h(\Omega)$ such that
\begin{equation}\label{con_weak_form}
\int_{\Omega} \beta  \nabla p_h \cdot \nabla v_h\, \, d x = \int_{\Omega} f v_h \, dx, \quad \forall v_h \in \widehat{S}_h(\Omega).
\end{equation}
Recently, it is observed that the partially penalized IFEM is of optimal convergence in the energy norm \cite{Lin2015partially}, which will be adopted in our paper (see Section \ref{egifem} for detailed description). To clarify the proposed conservative IFEM formulation, we state the definition of local and global conservation.
\begin{definition}[Local Conservation]
Given a triangulation $\mathcal{T}_h$, we say that the discrete flux ${\bf{u}}_h$ is conservative if the following holds true:
\begin{equation}\label{cons}
\int_{\partial T} {\bf{u}}_h \cdot {\bf{n}} \, ds = \int_T f \, dx, \quad\forall T \in \mathcal{T}_h,
\end{equation}
where ${\bf{n}}$ is the unit outward normal vector to $\partial T$. The corresponding global conservation is with $T$ replaced by $\Omega$ in the equation, \eqref{cons}.
\end{definition}
We would like to remark that the conservation is dependent on the choice of triangulations, $\mathcal{T}_h$. The issue with the conservation has drawn a lot of attention in literatures (see \cite{sun2009locally,lee2015locally} and references cited therein). Due to the absence of the piecewise constant in the IFEM space, the standard and its variant IFEMs are not locally or globally conservative.

\section{Enriched immersed finite element method (EIFEM)}\label{egifem}

In this section, we introduce the enriched immersed finite element method by the piecewise constant functions, that induces the local and global conservation.

\subsection{Enriched immersed finite element and its approximation property}

We enrich the standard IFEM space $\widehat{S}(\Omega)$ by piecewise constant functions. The idea has been introduced at \cite{sun2009locally,lee2015locally}. However, both of works are limited to the case when the mesh fits in the interface. This is the first attempt to introduce the enrichment for the interface problem where the interface is not necessarily aligned with the mesh. We shall denote by $E_h$, the enriched IFEM space, i.e., we define
\begin{equation}
E_h(\Omega) := \widehat{S}_h(\Omega) + C_h(\Omega),
\end{equation}
where
\begin{align*}
C_h(\Omega):=\{\psi \in L^2(\Omega)  \, | \,  \psi|_T \in \mathcal{P}^0(T) \}.
\end{align*}
To clarify the EIFEM formulation, we shall set the fixed unit normal vector for any given edge $e \in \mathcal{E}_h$, which will be denoted by ${\bf{n}}$. Note that the dimension of the space $\widehat{S}_h(\Omega)$ is the number of interior nodes for the triangulation $\mathcal{T}_h$, which will be denoted by $N_{0}$ while the dimension of the space $C_h$ is the number of elements, which will be denoted by $N_{e}$. We first, consider the space $H_h(\Omega) := \widetilde{H}^1(\Omega) + E_h(\Omega)$ and equip it with a broken $H^1$-norm:
\begin{equation}
\tnorm{\phi}_{h} := \sum_{T\in\mathcal{T}_h } \norm{\phi}^2_{1,T} + \sum_{e\in\mathcal{E}_h} \frac{1}{|e|} \int_e \jump{\phi}^2 \, ds, \quad \forall \phi \in H_h(\Omega).
\end{equation}
We introduce the scalar $L^2$ projection, $Q_e^0: H^{1/2}(e) \mapsto \mathcal{P}^0(e)$ and the vector $L^2$ projection, ${\bf{Q}}_e^0 : (H^{1/2}(e))^2 \mapsto (\mathcal{P}^0(e))^2$, defined, respectively by
\begin{equation}
Q_e^0(v) = \frac{1}{|e|} \int_e \av{v} \, ds \quad \mbox{ and } \quad {\bf{Q}}_e^0({\bf{v}}) = \frac{1}{|e|} \int_e \av{{\tend{v}}}\, ds.
\end{equation}
We note that $Q_h^0 : L^2(\Omega) \mapsto C_h$ denotes the standard $L^2$ projection onto $C_h$. In this section, we shall establish the optimal approximation property for EIFEM, both in the primal and flux variables. We begin with the EIFEM interpolation. Motivated by Lee et. al \cite{lee2015locally}, we introduce the EIFEM-interpolation operator $\Pi_h : \widetilde{H}^2(\Omega) \mapsto E_h$ as follows:
\begin{equation}
\Pi_h v = I_h v + Q_h^0(v - I_hv).
\end{equation}
\begin{lemma}
There holds the following estimate:
\begin{align}\label{C_I}
\tnorm{\phi - \Pi_h \phi}_h \lesssim h \norm{\phi}_{\wtilde{H}^2(\Omega)},
\end{align}
for all $\phi \in \widetilde{H}^2_{\Gamma_\beta}$.
\end{lemma}
\begin{proof}
We recall that
\begin{equation*}
\tnorm{\phi - I_h \phi}_h^2 = \sum_{T\in\mathcal{T}_h} \norm{\nabla (\phi-I_h \phi)}_{0,T}^2 + \sum_{e\in\mathcal{E}_h} \frac{1}{|e|} \int_e \jump{\phi - I_h \phi}^2
\, ds.
\end{equation*}
Since the estimate of the first term has been shown in Lemma \ref{C_pi}, we only need to investigate the second term. Invoking the standard trace inequality for $H^1$ function, we see that for $e \in \mathcal{E}_h$, let $T^+$ and $T^-$ be elements sharing $e$,
\begin{subeqnarray*}
\frac{1}{|e|} \int_e \jump{\phi - I_h \phi}^2_e \, ds &\lesssim& \frac{1}{|e|} \left( \int_e \left ( (\phi-I_h \phi)_{|_{T^+}} \right )^2  {\rm d} s + \int_e \left ( (\phi - I_h \phi)_{|_{T^-}}\right )^2  \, ds \right) \\
&\lesssim& \left( \frac{1}{h^2}\norm{\phi-I_h \phi}_{0,T^+ \cup T^-}^2 + |\phi-I_h \phi|_{1,T^+ \cup T^-}^2  \right).
\end{subeqnarray*}
Finally, for the remainder part, we first apply the standard inverse inequality on $T$ and trace theorem for the piecewise constant function, and then use the approximation property of $I_h$ to arrive at the conclusion. This completes the proof.
\end{proof}

\subsection{EIFEM formulation and its well-posedness}

We define the bilinear form $a_h(\cdot,\cdot) : H_h(\Omega) \times H_h(\Omega) \mapsto \Reals{}$ by, for all $v, w \in H_h(\Omega)$,
\begin{eqnarray*}
a_h(v, w) &=& \sum_{T \in \mathcal{T}_h} \int_T \beta \nabla v \cdot \nabla w \, dx - \sum_{e \in \mathcal{E}_h} \int_e \av{\beta \nabla v} \jump{w} \, d s \nonumber \\
&& + \theta \sum_{e \in \mathcal{E}_h^o} \int_e \av{\beta \nabla w} \jump{v} \, ds + \theta \sum_{e \in \mathcal{E}_h^\partial} \int_e \av{\beta \nabla w} \jump{Q_e^0(v)} \, ds \\
&& + \sum_{e \in \mathcal{E}_h^o} \frac{1}{|e|} \int_e \sigma (\beta) \jump{v}\jump{w} \, ds + \sum_{e \in \mathcal{E}_h^\partial}
\frac{1}{|e|} \int_e \sigma (\beta) \jump{Q_e^0(v)}\jump{Q_e^0(w)} \, ds, \label{bilinear_form}
\end{eqnarray*}
where $|e|$ is the measure of $e$, the symbol $\theta$ will be discussed below, and the symbol $\sigma(\beta)$ is to indicate that $\sigma$, the stabilization parameter is chosen depending on $\beta$ in each edge $e \in \mathcal{E}_h$. Theoretically, at the interface edge, $\sigma = \kappa \overline{\beta}/\underline{\beta}$ for some $\kappa > 0$ gives sufficient stabilization.

We are now in a position to state the enriched immersed finite element to solve \eqref{governing}. The enriched IFEM can then be formulated as follows: find $p_h \in E_h$ such that
\begin{equation}\label{weak_problem}
a_h(p_h, w_h) = (f, w_h)_{0}, \quad \forall w_h \in E_h.
\end{equation}
The bilinear form, here is the one that corresponds to interior penalty DG method, introduced in \cite{wheeler1978elliptic}. The symbol $\theta$ is the tuning parameter, which determines the type of interior penalty method, i.e., $\theta = -1, 0$ and $1$ result in NIPG, IIPG and SIPG, respectively \cite{sun2009locally,lee2015locally}.
\begin{remark}\label{remsingular}
The Dirichlet boundary condition has been imposed strongly for the space $\widehat{S}_h(\Omega)$. As such, the weak formulation \eqref{weak_problem} still results in weakly imposed zero boundary condition for the space $C_h$.
\end{remark}
We shall now list couple of simple but important observations for the problem \eqref{weak_problem}. We begin with the consistency.
\begin{lemma}
Suppose $p$ is the solution of (\ref{governing}) and $p_h$ is the solution of (\ref{weak_problem})
Then, we have
\begin{equation}\label{consis}
a_h(p-p_h, w_h)=0, \quad \forall w_h \in E_h.
\end{equation}
\end{lemma}
\begin{proof}
This follows from the definition of the $a_h(\cdot,\cdot)$ form. This completes the proof.
\end{proof}
For the coercivity of the bilinear form $a_h(\cdot,\cdot)$, we state and prove a simple, but important lemma:
\begin{lemma}\label{trace_like}
The following holds for all $\phi \in E_h(\Omega)$ and $T\in\mathcal{T}_h$ and edges $e$ of $T$.
\begin{align*}
\norm{\beta \nabla \phi \cdot  \bn_e}_{0,e}^2 \leq C_th^{-1} \norm{\beta \nabla \phi}_{0,T}^2.
\end{align*}
\end{lemma}
\begin{proof}
We decompose $\nabla \phi$ as
\begin{align*}
\nabla \phi
= (\nabla \phi \cdot \bn_\Gamma) \bn_\Gamma
+ (\nabla \phi \cdot \bt_\Gamma) \bt_\Gamma
:= \bw + \bz,
\end{align*}
where $\bn_\Gamma$ and $\bt_\Gamma$ are the unit normal and tangent vector to the interface $\Gamma$.
Since the functions in $\what{S}_h(T)$ satisfies the flux continuity condition, $\beta \bw \in H^1(T)$.
Also, $\nabla \phi $ has well defined trace on $\Gamma$, $\bz$ is in $H^1(T)$.
Thus, we have that
\begin{align}
\norm{\beta \bw \cdot \bn }_{0,e} &\leq Ch^{-1/2}  \norm{\beta \bw}_{0,T} \label{temp1000} \\
\norm{\bz  \cdot \bn }_{0,e} &\leq Ch^{-1/2}  \norm{\bz}_{0,T}. \label{temp1001}
\end{align}
By the triangular inequality and inequalities (\ref{temp1000}) and (\ref{temp1001}), we have
\begin{align*}
\norm{\beta \nabla \phi \cdot  \bn_e}_{0,e}
&\leq  \norm{\beta \bw \cdot \bn_e}_{0,e}
+ \norm{\beta \bz \cdot \bn_e}_{0,e} \\
&\leq  \norm{\beta \bw \cdot \bn_e}_{0,e}
+  \overline{\beta} \norm{\bz \cdot \bn_e}_{0,e} \\
&\leq Ch^{-1/2} ( \norm{\beta \bw}_{0,T}
+ \overline{\beta} \norm{\bz}_{0,T} ) \\
&\leq Ch^{-1/2} \left(1 + \frac{\overline{\beta}}{\underline{\beta}} \right)\norm{\beta \nabla \phi}_{0,T}. \quad \Box
\end{align*}
\end{proof}

We are in a position to establish the coercivity of the bilinear form $a_h(\cdot,\cdot)$.
\begin{lemma}\label{lemma:corc}
There exists some $\sigma_0>0$ such that the following holds whenever $\sigma>\sigma_0$,
\begin{align}\label{corc}
C_\alpha \tnorm{\phi_h}_h^2 \leq a_h(\phi_h,\phi_h), \quad  \forall \phi_h \in E_h(\Omega) ,
\end{align}
for some $\alpha>0$.
\end{lemma}
\begin{proof}
Using the Cauchy's inequality, we have that
\begin{eqnarray*}
&& \sum_{e \in \mathcal{E}_h} \int_e| \av{\beta \nabla \phi_h \cdot \bn_e} \jump{\phi_h}| \, ds  \\
&& \qquad \leq \left( h \sum_{e \in \mathcal{E}_h}  \norm{ \av{\beta \nabla \phi_h \cdot \bn_e}}_{0,e}^2  \right)^{\frac{1}{2}}
\left(  h^{-1} \sum_{e \in \mathcal{E}_h} \norm{ \jump{\phi_h}}_{0,e}^2  \right)^{\frac{1}{2}}.
\end{eqnarray*}
Let $T^+$ and $T^-$ be neighboring elements of the edge $e$. By applying the Lemma \ref{trace_like} and using the fact that there are at most finite number of neighboring elements for any given element of the mesh, we have that
\begin{subeqnarray*}
h \sum_{e \in \mathcal{E}_h}  \norm{\av{\beta \nabla \phi_h \cdot \bn_e}}_{0,e}^2 &\lesssim& h \sum_{e \in \mathcal{E}_h} \left( \norm{(\beta \nabla \phi_h)_{|_{T^+}} \cdot \bn_e}_{0,e}^2  + \norm{(\beta \nabla \phi_h)_{|_{T^-}} \cdot \bn_e}_{0,e}^2 \right)   \\
&\lesssim& \sum_{e \in \mathcal{E}_h}  \norm{\beta \nabla \phi_h}_{0,T^+ \cup T^-}^2 \lesssim \overline{\beta} \sum_{T \in \mathcal{T}_h}  \norm{\nabla \phi_h}_{0,T}^2.
\end{subeqnarray*}
Invoking Young's inequality, for $\delta > 0$, we have that
\begin{subeqnarray*}
&&(1 - \theta) \sum_{e \in \mathcal{E}_h} \int_e \left | \av{\beta \nabla \phi_h \cdot \bn_e} \jump{\phi_h} \right | ds \\
&& \qquad \lesssim \frac{\delta}{2} \sum_{T \in \mathcal{T}_h}  \norm{\nabla \phi_h}_{0,T}^2 + \frac{(1-\theta)^2 \overline{\beta}  }{2\delta} \sum_{e \in \mathcal{E}_h} \frac{1}{|e|} \norm{\jump{\phi_h}}_{0,e}^2.
\end{subeqnarray*}
Thus, we have
\begin{subeqnarray*}
&& a_h(\phi_h,\phi_h) = \sum_{T \in\mathcal{T}_h} \int_T \beta \nabla \phi_h \cdot \nabla \phi_h \, dx \\
&&\quad -(1 - \theta) \sum_{e \in \mathcal{E}_h} \int_e \av{\beta \nabla \phi_h \cdot \bn_e} \jump{\phi_h} \,ds
+ \sum_{e\in \mathcal{E}_h} \frac{1}{|e|} \int_e \sigma(\beta) \jump{\phi_h}^2 \, ds \\
&&\quad \gtrsim \left(\underline{\beta}-\frac{\delta}{2} \right)\sum_{T\in\mathcal{T}_h} \norm{\nabla \phi_h}_{0,T}^2
+ \left( C \min_{e \in \mathcal{E}_h} \sigma(\beta) - \frac{(1-\theta)^2 \overline{\beta}}{2\delta} \right) \sum_{e \in \mathcal{E}_h} \frac{1}{|e|}  \norm{ \jump{\phi_h}}_{0,e}^2,
\end{subeqnarray*}
for some generic constant $C > 0$. By choosing $\delta = \underline{\beta}$ and $\min_{e \in \mathcal{E}_h} \sigma(\beta)$ large enough, we obtain the desired result. This completes the proof.
\end{proof}

The continuity of the bilinear form $a_h(\cdot,\cdot)$ can be proven by the same techniques used in the proof of Lemma \ref{lemma:corc}.
\begin{lemma}
There exists some $C_b$ such that the following holds when $\sigma>0$,
\begin{align}\label{continuity}
a_h(\phi_h,\psi_h) \leq C_b \tnorm{\phi_h}_h  \tnorm{\psi_h}_h  , \quad  \forall \phi_h, \psi_h  \in E_h(\Omega).
\end{align}
\end{lemma}

We now state and prove the error estimate for the primary variable in $\tnorm{\cdot}_h$-norm.
\begin{theorem}
Let $p$ be the solution of \eqref{governing} and $p_h$ be the solution of (\ref{weak_problem}).
Then there exists some $C>0$ such that following holds.
\begin{align}\label{energy_norm}
\tnorm{p-p_h}_h \leq Ch \norm{f}_{L^2(\Omega)}.
\end{align}
\end{theorem}
\begin{proof}
By triangular inequality, we have
\begin{align}\label{h1_1}
\tnorm{p-p_h}_h \leq \tnorm{p_h - \Pi_h p}_h + \tnorm{p - \Pi_h p }_h.
\end{align}
From inequalities (\ref{corc}), (\ref{continuity}) and Ce\'{a}'s Lemma, it follows that
\begin{align}\label{h1_2}
\tnorm{p_h - \Pi_h p}_h \leq \frac{C_b}{C_\alpha} \norm{p - \Pi_h p}_h.
\end{align}
By (\ref{C_I}), (\ref{h1_1}), (\ref{h1_2}) and (\ref{C_R}) we have,
\begin{align*}
\tnorm{p-p_h}_h
&\leq \left(1+\frac{C_b}{C_\alpha}\right)C_I h \norm{p}_{\widetilde{H}^2(\Omega)} \\
&\leq Ch \norm{f}_{L^2(\Omega)}. \quad \Box
\end{align*}
\end{proof}
Finally, we state the error estimate in $L^2$-norm.
\begin{theorem}
Let $p$ be the solution of \eqref{governing} and $p_h$ be the solution of (\ref{weak_problem}). Suppose $\theta = -1$ in (\ref{bilinear_form}).
Then there exists some $C>0$ such that following holds.
\begin{align}\label{L2_norm}
\norm{p-p_h}_{L^2(\Omega)} \leq Ch \norm{f}_{L^2(\Omega)}.
\end{align}
\end{theorem}
\begin{proof}
This can be proven by the standard duality argument together with (\ref{energy_norm}). This completes the proof.
\end{proof}

\section{Conservative flux reconstruction and its error analysis}

In this section, we discuss the flux reconstruction. Unlike the prior works \cite{sun2009locally,lee2015locally} (see also $H({\rm div})$-flux reconstruction of DG developed and analyzed in e.g., \cite{Ern2007accurate}), the jump discontinuity over $\Gamma$ that are not necessarily aligned with the grid requires to carefully define the flux along the edge to preserve the conservation as well as to produce certain accuracy. More precisely,
The EG-flux reconstruction introduced in \cite{sun2009locally,lee2015locally} was given as follows:
\begin{equation*}
\mathbf{u}_h \cdot \bn_e = -   \av{\beta \nabla p_h \cdot \bn_e} + \frac{\sigma(\beta)}{|e|} \jump{p_h}.
\end{equation*}
On the other hand, for the case when the discontinuity is allowed within an element, we modify this as follows, which will be coined as ``EIFEM-flux recovery". We shall define $\bu_h$ so that it belongs to the lowest order Raviart-Thomas (RT) space \cite{raviart1977mixed} by assigning its degree of freedom in each edge $e \in \mathcal{E}_h$ as follows:
\begin{equation}\label{vel_recovery}
\mathbf{u}_h \cdot \bn_e := \frac{1}{|e|} \int_e \left( - \av{\beta \nabla p_h \cdot \bn_e} + \frac{\sigma(\beta)}{|e|} \jump{p_h} \right)  \, ds,
\end{equation}
We now show that the EIFEM-flux recovery possesses the local and global conservation property:
\begin{proposition}
The flux ${\bf{u}}_h$ defined through \eqref{vel_recovery} satisfies the local and global conservation, namely,
\begin{equation}\label{loc_mass}
\int_{\partial T} \mathbf{u}_h\cdot \bn \, ds = \int_T f \, dx, \quad \forall T \in \mathcal{T}_h,
\end{equation}
where $\bn$ is the unit outward normal to $\partial T$, and
\begin{equation*}
\int_{\partial \Omega} \mathbf{u}_h\cdot \bn \, ds = \int_\Omega f \, dx,
\end{equation*}
where $\bn$ is the unit outward normal to $\partial \Omega$.
\end{proposition}
\begin{proof}
Let $T^+ \in \mathcal{T}_h$ be given. By taking the test function $w_h=1$ on $T^+$ and 0 elsewhere, for the equation \eqref{weak_problem}, we have that
\begin{eqnarray*}
- \sum_{e \subset \partial T^+} \int_e \av{\beta \nabla p_h \cdot \bn_{e} }_e \, ds
+ \sum_{e \subset \partial T^+} \frac{1}{|e|} \int_e \sigma(\beta) (p_h|_{T^+} - p_h|_{T^-}) \, ds = \int_{T^+} f \, dx,
\end{eqnarray*}
where $T^-$ is an element adjacent to $T^+$, sharing $e$ as a common edge. By the definition of $\mathbf{u}_h$ in (\ref{vel_recovery}) and by the above identity, we see that
\begin{subeqnarray*}
\int_{\partial T^+} \mathbf{u}_h\cdot \bn  \, d x &=& \int_{\partial T^+} (\mathbf{u}_h\cdot \bn_e ) (\bn_e  \cdot \bn ) \, d s  = - \sum_{e \subset \partial T^+} \int_e \av{\beta \nabla p_h \cdot \bn} \, ds \\
&& \quad+ \sum_{e \subset \partial T^+} \frac{1}{|e|} \int_e \sigma(\beta) (p_h|_{T^+}- p_h|_{T^-}) \, ds = \int_{T^+} f \, dx.
\end{subeqnarray*}
This establishes the local conservation. Now, by taking $w_h = 1$, globally on $\Omega$, i.e., by summing over $T \in \mathcal{T}_h$, we are led to
\begin{align*}
\int_{\partial \Omega} \mathbf{u}_h \cdot \bn \, ds = \int_\Omega f \, dx.
\end{align*}
This completes the proof.
\end{proof}

We are now in a position to state and prove the error estimates of the flux recovery for the EIFEM. The error estimate will be provided for both $\|{\bf{u}} - {\bf{u}}_h\|_0$ and $\|{\rm div}({\bf{u}} - {\bf{u}}_h)\|_0$.
\begin{theorem}
Let $\bu$ be the solution of \eqref{governing} and $\bu_h$ be the EIFEM-flux computed by (\ref{vel_recovery}).
If $\bu \in (H^1(\Omega))^2$, then the following holds.
\begin{equation*}
\norm{\bu-\bu_h}_{L^2(\Omega)} \leq Ch  \norm{f}_{L^2(\Omega)}.
\end{equation*}
\end{theorem}
\begin{proof}
For any given $e \in \mathcal{E}_h^o$, we consider the triangle that shares it as a common edge, say $T^+$ and $T^-$. We observe that from the definition of
$\bu_h$ and by the fact that $p \in \widetilde{H}^2(\Omega)$ and $\bu \in H^1(\Omega)$, we have that
\begin{subeqnarray*}
(\mathbf{u}_h + {\bf{Q}}_e^0 (\beta \nabla p_h)) \cdot \bn_e &=& \frac{1}{|e|} \int_e \jump{\beta \nabla p_h \cdot \bn_e } \, ds + \frac{1}{|e|}  \int_e \sigma(\beta) \jump{p_h} \, ds \\
&=& \frac{1}{|e|} \int_e \jump{\beta \nabla (p_h -p) \cdot \bn_e} \, ds + \frac{1}{|e|} \int_e \sigma(\beta) \jump{p_h- p} \, ds,
\end{subeqnarray*}
Thus,
\begin{subeqnarray*}
\int_e (\mathbf{u}_h + {\bf{Q}}_e^0(\beta \nabla p_h) \cdot \bn_e)^2  \, ds  &=& \int_e \jump{\beta \nabla (p_h -p) \cdot \bn_e} (\mathbf{u}_h + {\bf{Q}}_e^0 (\beta \nabla p_h))) \cdot \bn_e \, ds \\
&& +  \frac{1}{|e|}  \int_e \sigma(\beta) \jump{p_h- p} (\mathbf{u}_h + {\bf{Q}}_e^0(\beta \nabla p_h)) \cdot \bn_e \, ds.
\end{subeqnarray*}
Applying the Cauchy Schwarz and triangule inequality, we have
\begin{equation*}
\norm{((\mathbf{u}_h + {\bf{Q}}_e^0(\beta \nabla p_h)) \cdot \bn_e }_{0,e} \lesssim \norm{ \jump{\beta \nabla (p_h -p) \cdot \bn_e }}_{0,e} + \frac{1}{|e|} \norm{\jump{p_h - p}}_{0,e}.
\end{equation*}
By the estimates for $p-p_h$ in (\ref{energy_norm}) and (\ref{L2_norm}) together with the trace inequality, we have
\begin{align}\label{temp100}
\norm{(\mathbf{u}_h + {\bf{Q}}_e^0(\beta \nabla p_h))\cdot \bn_e }_{0,e} \lesssim h^{1/2} \norm{p}_{\wtilde{H}^2(T^+ \cup T^-)}.
\end{align}
Furthermore, we note that since $\bu = -\beta\nabla p \in (H^1(\Omega))^2$,
\begin{subeqnarray*}
\norm{(\beta \nabla p_h -  {\bf{Q}}_e^0(\beta \nabla p_h)) \cdot \bn_e }_{0,e} &\lesssim& \norm{(\beta \nabla p_h - \beta \nabla p)\cdot \bn_e}_{0,e}
\\
&& + \norm{(\beta \nabla p -  {\bf{Q}}_e^0(\beta \nabla p_h))\cdot \bn_e}_{0,e} \\
&& + \norm{  {\bf{Q}}_e^0(\beta \nabla p - \beta \nabla p_h) \cdot \bn_e}_{0,e} \nonumber  \\
&\lesssim& h^{1/2}\norm{p}_{\wtilde{H}^2(T^+\cup T^-)} + \norm{(\beta \nabla p - \beta \nabla p_h) \cdot \bn_e}_{0,e} \nonumber  \\
&\lesssim& h^{1/2}\norm{p}_{\wtilde{H}^2(T^+\cup T^-)}.
\end{subeqnarray*}
These result in
\begin{align*}
\norm{(\mathbf{u}_h +\beta \nabla p_h)\cdot \bn_e}_{0,e} \lesssim h^{1/2} \norm{p}_{\wtilde{H}^2(T^+ \cup T^-)}.
\end{align*}
For a given $T \in \mathcal{T}_h$, we apply the scaling argument to obtain that
\begin{eqnarray*}
\norm{\mathbf{u}_h +\beta \nabla p_h}_{0,T} &\lesssim& h^{1/2} \sum_{e \in\partial T} \norm{(\mathbf{u}_h + \beta \nabla p_h)\cdot \bn_e }_{0,e} \nonumber  \\
&\lesssim& h \norm{p}_{\widetilde{H}^2(\mathcal{M}_h)}, \label{temp102}
\end{eqnarray*}
where $\mathcal{M}_h = \{ T_k \in \mathcal{T}_h : T_k \cap T \neq \emptyset\}$. Finally, by applying the triangle inequality, (\ref{temp102}) and (\ref{energy_norm}), we obtain that
\begin{subeqnarray*}
\norm{\mathbf{u}_h - \mathbf{u}}_{0,\Omega} \lesssim \sum_{T \in \mathcal{T}_h} \left( \norm{\mathbf{u}_h + \beta \nabla p_h}_{0,T} + \norm{\beta \nabla p_h - \beta \nabla p}_{0,T} \right ) \lesssim h \norm{p}_{\widetilde{H}^2(\Omega)}.
\end{subeqnarray*}
We obtain the desired inequality by (\ref{C_R}). This completes the proof.
\end{proof}
We now establish the error estimate of $\bu- \bu_h$ in $H({\rm div})$-norm.
\begin{theorem}
Let $\bu$ be the solution of \eqref{governing} and $\bu_h$ be the EIFEM-flux. Assume that $f \in H^1(\Omega)$, then it holds that
\begin{equation*}
\sum_{T \in \mathcal{T}_h }\norm{ {\rm div} (\bu-\bu_h)}_{0,T} \lesssim h \norm{f}_{1,\Omega}.
\end{equation*}
\end{theorem}
\begin{proof}
Let $T \in \mathcal{T}_h$. By the divergence theorem and the local mass conservation in (\ref{loc_mass}), we have
\begin{equation*}
\int_T  {\rm div} \bu_h \, dx = \int_{\partial T} \bu_h  \cdot \bn \, ds = \int_T f \, dx.
\end{equation*}
Namely, we see that ${\rm div} \bu_h$ is a local average of $f$ on $T$. Thus, by the Bramble-Hilbert Lemma,
\begin{align*}
\norm{ {\rm div}( \bu-\bu_h) }_{0,T}
= \norm{ f - Q_h^0(f)}_{0,T} \lesssim h \norm{f}_{1,T}.
\end{align*}
By summing over all elements $T$, and by (\ref{C_R}), we have the desired inequality. This completes the proof.
\end{proof}

\section{Auxiliary space preconditioner for EIFEM}

In this section, we give a description of preconditioning techniques based on fictitious or auxiliary spaces as pioneered in \cite{nepomnyaschikh1991decomposition,xu1996auxiliary}. We then establish that the abstract framework can be applied for designing the auxiliary space preconditioner for EIFEM.

\subsection{Auxiliary space preconditioner}
Let $V$ be a real Hilbert space with inner product $a(\cdot,\cdot)$ and energy norm $\|\cdot\|_A$. The fictitious space method solves the following linear system: find $u \in V$ for
\begin{equation}
a(u,v) = f(v), \quad \forall v \in V.
\end{equation}
The main building blocks are
\begin{itemize}
\item a fictitious space $\overline{V}$, i.e., another real Hilbert space equipped with the inner product $\overline{a}(\cdot,\cdot)$, which induces the norm $\|\cdot\|_{\overline{A}}$.
\item a continuous and surjective linear transfer operator $\Pi : \overline{V} \mapsto V$.
\end{itemize}

We tag dual spaces by $'$, adjoint operators by $*$, and use angle brackets for duality pairings and write $A : V \mapsto V'$ and $\overline{A} : \overline{V} \mapsto \overline{V}'$ for operator form of bilinear map $a(\cdot,\cdot)$ and $\overline{a}(\cdot,\cdot)$, respectively. The fictitious space preconditioner is then given by
\begin{equation}
B = \Pi \circ \overline{A}^{-1} \circ \Pi^* : V' \mapsto V.
\end{equation}
It is well-known that the aforementioned operator $B$ is actually positive definite (Lemma 2.1, \cite{hiptmair2007nodal}). We state the fictitious space lemma and provide the elementary proof \cite{nepomnyaschikh1991decomposition}.
\begin{theorem}[Fictitious Space Lemma]\label{fictitious}
Assume that $\Pi$ is surjective, and
\begin{subeqnarray}
&&\exists c_0 > 0 \mbox{ such that } \forall v \in V, \,\, \exists \overline{v} \in \overline{V} \mbox{ with } v = \Pi \overline{v} \mbox{ and } \|\overline{v}\|_{\overline{A}} \leq c_0 \|v\|_A \qquad \quad \\
&&\exists c_1 > 0 \mbox{ such that } \|\Pi \overline{v}\|_{A} \leq c_1 \|\overline{v}\|_{\overline{A}}, \quad \forall \overline{v} \in \overline{V}.
\end{subeqnarray}
Then, we have
\begin{equation}
c_0^{-2} \|v\|_A^2 \leq a(BA v,v) \leq c_1^2 \|v\|_A^2, \quad v \in V.
\end{equation}
\end{theorem}
This will immediately lead to an estimate for the spectral condition number of the operator $BA$ as follows:
\begin{equation}
\kappa(BA) = \frac{\lambda_{max}(BA)}{\lambda_{min}(BA)} \leq (c_0 c_1)^2.
\end{equation}

The auxiliary space method is a general preconditioning approach based on a relaxation scheme and an auxiliary space pioneered by Xu \cite{xu1996auxiliary}. The feature of the auxiliary space approach lies in the choice of the following auxiliary space:
\begin{equation}
\overline{V} := V \times W_1 \times \cdots \times W_J,
\end{equation}
where $V$ as a component of $\overline{V}$ is equipped with an inner product $s(\cdot,\cdot)$, different from the originally given bilinear form $a(\cdot,\cdot)$, and $W_1,\cdots, W_J$ are Hilbert spaces endowed with inner products $a_j(\cdot,\cdot)$ for $j=1, \cdots, J$. The operator $S : V \mapsto V'$ induced by $s(\cdot,\cdot)$ on $V$ is typically called the smoother. Under this setting, the auxiliary space method adopts the fictitious space approach with the inner product: $\forall \overline{v} = (v, v_1, \cdots, v_j), \overline{w} = (w, w_1, \cdots, w_j) \in \overline{V}$,
\begin{equation}
\overline{a}(\overline{v}, \overline{w}) := s(v,w)  + \sum_{j=1}^J {a}_j (v_j, w_j).
\end{equation}
Furthermore, we introduce a linear transfer operator $\Pi_j : W_j \mapsto V$, for each $W_j$, with $\Pi_0 = I$, from which we build the surjective transfer operator:
\begin{equation}
\Pi := \Pi_0 \times \Pi_1 \times \cdots \times \Pi_J : \overline{V} \mapsto V,
\end{equation}
whose action is given as follows:
\begin{equation}
\Pi \overline{v} = \Pi (v, w_1, w_2, \cdots, w_J) = v + \sum_{j=1}^J \Pi_j w_j \in V.
\end{equation}
This will lead to the construction of the auxiliary space preconditioner given as follows:
\begin{equation}
B := S^{-1} + \sum_{j=1}^J \Pi_j \circ A_j^{-1} \circ \Pi_j^*,
\end{equation}
where $A_j$'s are operators that correspond to the bilinear form $a_j(\cdot,\cdot)$ for $j=1,\cdots,J$. The verification of the assumption of the Theorem \ref{fictitious} boils down to the following three steps:
\begin{theorem}\label{main:aux}
Assume that there hold:
\begin{itemize}
\item there exists $c_j > 0$ for norms of the transfer operators $\Pi_j$:
\begin{subeqnarray}\label{est1}
\|\Pi_j w_j \|_A^2 \leq c_j a_j(w_j, w_j), \quad \forall w_j \in W_j,
\end{subeqnarray}
\item the boundedness of $S^{-1}$, i.e., there exists $c_S > 0$ such that
\begin{subeqnarray}
\|v\|_{A}^2 &\lesssim& c_S \|v\|_S^2, \quad \forall v \in V,
\end{subeqnarray}
\item for every $v \in V$, there are $v_0 \in V$ and $w_j \in W_j$ such that $v = v_0 + \sum_{j=1}^J \Pi_j w_j$ and for some $c_0 > 0$
\begin{subeqnarray}
s(v_0,v_0) + \sum_{j=1}^J a_j(w_j,w_j)  \leq c_0^2 \|v\|_{A}^2.
\end{subeqnarray}
\end{itemize}
Then it holds true that
\begin{equation}
\kappa(B A) \leq c_0^2 (c_S^2 + c_1^2 + \cdots + c_J^2).
\end{equation}
\end{theorem}
\begin{remark}
The aforementioned Theorem \ref{main:aux} can be shown to hold even if the bilinear forms $a_j$ on the auxiliary space $W_j$ are replaced by any spectrally equivalent bilinear forms, $b_j$, namely, we can use the preconditioner for the operator $A_j$.
\end{remark}

\subsection{Auxiliary space preconditioner for solving EIFEM}

We let
\begin{subeqnarray}
V &=& E_h \\
W_1 &=& \widehat{S}_h(\Omega) = {\rm span}\{ \phi_j \}_{j=1}^{N_0} \\
W_2 &=& C_h(\Omega)={\rm span}\{ \psi_j \}_{j=1}^{N_e},
\end{subeqnarray}
where $\phi_j$ is a nodal linear basis for the IFEM and $\psi_j$ is the element-wise constant function defined by $\psi_h|_{T_\ell} = \delta_{j\ell}$.
Here, $N_0$ is the number of nodes in $\mathcal{T}_h$ and $N_e$  is the number of elements in $\mathcal{T}_h$.
The system arising from EIFEM (\ref{weak_problem}) is written in $N$ by $N$ ($N = N_0 + N_e$) system
\begin{equation}
\tenq{A} \tend{u} = \tend{f}
\end{equation}
where the matrix can be written as
\begin{equation}
\tenq{A} =
\left(
\begin{array}{cc}
\tenq{A}_{11} & \tenq{A}_{12} \\
\tenq{A}_{21} & \tenq{A}_{22}
\end{array}
\right) .
\end{equation}
The submatrices are
\begin{subeqnarray*}
\tenq{A}_{11}(i,j) &=& \sum_{T \in \mathcal{T}_h} \int_T \beta \nabla \phi_j \cdot \nabla \phi_i dx
- \sum_{e \in \mathcal{E}_h^o} \int_e \av{ \beta \nabla \phi_i \cdot \bn_e}  \jump{\phi_j} \, ds \\
&&+ \theta \sum_{e \in \mathcal{E}_h^o} \int_e \av{ \beta \nabla \phi_j \cdot \bn_e } \jump{\phi_i} \, ds +
\frac{1}{|e|} \sum_{e \in \mathcal{E}_h^o} \int_e  \sigma(\beta) \jump{\phi_j} \jump{\phi_i}  \, ds \\
\tenq{A}_{12}(i,j) &=& \theta \sum_{e} \int_e \av{ \beta \nabla \phi_j \cdot \bn_e } \jump{\psi_i} \, ds, \\
\tenq{A}_{21}(i,j) &=& -\sum_{e \in \mathcal{E}_h^o} \int_e \av{ \beta \nabla \psi_i \cdot \bn_e } \jump{\phi_j}_e \, ds, \\
\tenq{A}_{22}(i,j) &=& \frac{1}{|e|} \sum_{e} \int_e  \sigma(\beta) \jump{\psi_j} \jump{\psi_i} \, ds.
\end{subeqnarray*}
We note that the Dirichlet boundary condition is imposed strongly for linear piecewise element space. We provide a remark on its effect.
\begin{remark}
If we impose the Dirichlet boundary condition weakly for the space $E_h$, then due to the redundancy of the constant function from both spaces $W_1$ and $W_2$, a degeneracy occurs, i.e.,
$$
W_1 \cap W_2 = {\rm span} \{1\}.
$$
This results in the system matrix singular. In our formulation, the strong Dirichlet boundary condition has been imposed on $\widehat{S}_h(\Omega)$, while the  zero Dirichlet boundary is imposed weakly for $C_h$. As such, the resulting system becomes nonsingular.
\end{remark}
We shall consider the following auxiliary space decomposition:
\begin{equation}
\overline{V}= V \times W_1 \times W_2.
\end{equation}
We use the notation $\tend{x} = \tend{x}_1+\tend{x}_2$ where $\tend{x}_1$ belongs to $\mathcal{W}_1$, the vector representation of $W_1$ and $\tend{x}_2$ belongs to $\mathcal{W}_2$, the vector representation of $W_2$. Therefore, the corresponding auxiliary space preconditioner we propose in this paper, consists of the following three steps: step i) pre-smoothing step ii) solving each diagonal block system, $\tenq{A}_{11}$ and $\tenq{A}_{22}$,
which are the restrictions of $\tenq{A}$ to $W_1$ and $W_2$, respectively, and step iii) post-smoothing (for symmetrization).
\begin{algorithm}[Preconditioner]
We apply the following in each iteration:
\begin{enumerate}
\item Gauss Seidel $N_{GS}$.
\item Compute Residual. $\tend{R} = \tend{R}_1+\tend{R}_2$.
\item Precondition for each submatrix $\tenq{A}_{11}$ and $\tenq{A}_{22}$:
  $\tend{z}_1=AMG(\tenq{A}_{11}) \tend{R}_1$ and $\tend{z}_2 = AMG(\tenq{A}_{22}) \tend{R}_2$.
\item Update corrections:
  $\tend{x}=(\tend{x}_1+\tend{z}_1)+(\tend{x}_2+\tend{z}_2)$.
\item Backward Gauss-Seidel $N_{GS}$.
\end{enumerate}
\end{algorithm}

We equip three spaces, $V, W_1$ and $W_2$ with inner products as follows. Starting at $W_1$ and $W_2$, we introduce ${a}_j(\cdot,\cdot) : W_j \times W_j \mapsto \Reals{}$ as the restriction of $a(\cdot,\cdot)$ onto $W_j$ for $j=1,2$, respectively. Note that since $W_j \subset V$, it holds that for $j = 1,2$,
\begin{equation}
{a}_j(v_j,w_j) = a(v_j,w_j), \quad \forall v_j, w_j \in W_j.
\end{equation}
For the space $V$, let $\tenq{D}$ be the diagonal part of the matrix $\tenq{A}$. We then define the operator $s(\cdot,\cdot) : V\times V\mapsto \Reals{}$ by the following relation:
\begin{equation}
s(v,w) = \tend{v}^T \tenq{D} \tend{w} \quad \forall v, w \in V,
\end{equation}
where $\tend{v}$ and $\tend{w}$ are the representation of $v$ and $w$ in $\Reals{N}$. We now equip the space $\overline{V}$ with the inner product defined as follows:
\begin{equation}
\overline{a}(\overline{v},\overline{v}) := s(v_0,v_0) + {a}_1(v_1,v_1) + {a}_2(v_2,v_2), \quad \forall \overline{v} = (v_0,v_1,v_2) \in \overline{V}.
\end{equation}
For $i=1,2$, we introduce an operator $\Pi^*_i : V \mapsto W_i$ defined as simple injections, i.e., $\Pi_i^* v = v_i$, $\forall v = v_1 + v_2 \in V$, with
$v_i \in W_i$ for $i=1,2$, and inclusion maps, $\Pi_i :  W_i \mapsto V$. $\Pi : \overline{V} \mapsto V$ can be defined by
\begin{equation}
\Pi \overline{v} = v_0 + \Pi_1 v_1 + \Pi_2 v_2, \quad \forall \overline{v} = (v_0, v_1,v_2) \in \overline{{V}}.
\end{equation}
We remark that $\Pi_1$ and $\Pi_2$ are simply the identity which are simpler than those for DG.
The preconditioner can be stated as follows:
\begin{equation}
{B} = {S}^{-1} + \Pi_1 \circ {A}_1^{-1}\circ \Pi_1^* + \Pi_2 \circ {A}_2^{-1}\circ \Pi_2^*,
\end{equation}
To establish the quality of the preconditioner ${B}$, we shall need to establish the estimate \eqref{est1}, but this is trivial.
Secondly, we shall prove
\begin{lemma}
For any $v \in V$, we have that
\begin{equation}\label{mgcond1}
\|v\|_A^2 \lesssim c_S \|v\|_S^2.
\end{equation}
\end{lemma}
\begin{proof}
We note that $\|v\|_A^2 := A(v,v), \,  \|v\|_S^2 := S(v,v)$. This inequality \eqref{mgcond1} is due to the Cauchy-Schwarz and inverse inequality. This completes the proof.
\end{proof}
We now state and prove the last step:
\begin{lemma}\label{mainest}
For all $v \in V$, there exist $v_0 \in V$, $v_1 \in W_1$  and $v_2 \in W_2$ such that $v = v_0 + v_1 + v_2$ and
\begin{equation}
s(v_0,v_0) + \|v_1\|_A^2 + \|v_2\|_{A}^2 \leq c_0^2  \|v\|_A^2.
\end{equation}
\end{lemma}
\begin{proof}
Given $v \in V$, we define $v_j \in W_j$ for $j = 1,2$, by the solution to the following equation:
\begin{equation}\label{s11block}
(v_j,w_j)_{A} = (v, w_j)_{A}, \quad \forall w_j \in W_j.
\end{equation}
Then, it is immediate that $\|v_j\|_{A} \lesssim \|v\|_{A}$ for $j = 1, 2$. We now define $v_0$ by
\begin{equation}
v_0 = v - v_1 - v_2 \in V,
\end{equation}
then, it is enough to show that
$$
\|v_0\|_{S} \lesssim \|v\|_{A}.
$$
We note that
\begin{eqnarray*}
\|v_0\|_{S} = \|v - v_1 -v_2\|_{S} \lesssim \|v - v_1\|_S + \|v_2\|_{S} \lesssim \|v - v_1\|_S + \|v\|_{A}.
\end{eqnarray*}
Therefore, we shall show that $\|v - v_1\|_S \lesssim \|v\|_{A}$. Let $\chi_{_E}$ be a characteristic function, which is one on any subset $E \subset \Omega$ and zero elsewhere. We then consider the adjoint problem to find $\psi \in \widetilde{H}_0^1(\Omega)$ such that with $\beta_T = \beta|_T$ for any $T \in\mathcal{T}_h$,
\begin{equation}
- \nabla \cdot \beta \nabla \psi =  D_h(v-v_1)  \,\, \mbox{ in } \Omega, 
\end{equation}
subject to the interface conditions and homogeneous boundary condition on $\partial \Omega$. Here $D_h: V \mapsto V$ is the positive operator that satisfies the following relation:
$$
(D_h v,v)_0 = (D_h^{1/2} v, D_h^{1/2}v)_0 = h^2 s(v,v), \quad \forall v \in V.
$$
Then we have that for all $\mu \in W_1$,
\begin{subeqnarray*}
\|D_h^{1/2}(v-v_1)\|^2_0 &=& (v - v_1, \psi)_{A} \\
&=& (v - v_1, \psi - \mu)_{A}, \\
&\lesssim& \|v - v_1\|_{A} \|\psi - \mu\|_{A} \\
&\lesssim& h \|v - v_1\|_A \|D_h^{1/2}(v-v_1)\|_0,
\end{subeqnarray*}
where the last inequality is due to the elliptic regularity stated in Proposition \ref{main:them}. This gives that
\begin{subeqnarray*}
\|v-v_1\|_S^2 = h^{-2}\|D_h^{1/2}(v - v_1)\|_0^2 \lesssim \|v - v_1\|^2_{A}.
\end{subeqnarray*}
This completes the proof.
\end{proof}

In the remaining section, we show that the block matrix $\tenq{A}_{22}$ can be easily solved by a classical algebraic multigrid method. We recall that a matrix
$\tend{M} = (M_{ij})$ is an M-matrix if it is irreducible, i.e., the graph corresponding to $C$ is connected and the following conditions hold:
\begin{subeqnarray}
&& {M}_{jj} > 0, \quad \forall 1 \le j \le n, \\
&& {M}_{ij} \le 0, \quad \forall i,j: i \neq j, \\
&& {M}_{jj} \ge \sum_{i =1: \neq j}^n |M_{ij}|, \quad \forall 1 \le j \le n, \\
&& {M}_{jj} > \sum_{i =1: \neq j}^n |M_{ij}|, \quad \mbox{ for at least one }j.
\end{subeqnarray}

The following Lemma indicates that the block matrix $\tenq{A}_{22}$ can be easily handled by a classical algebraic multigrid method.
\begin{lemma}
The matrix $\tenq{A}_{22}$ is M-matrix and weakly diagonally dominant.
\end{lemma}
\begin{proof}
It is immediate to see that the time derivative term restricted on $W_2$ is positive diagonal and
\begin{equation}
\tenq{A}_{22}(i,i) =  \sum_{e \in \mathcal{E}_h} \frac{1}{|e|} \int_e \sigma(\beta) \jump{\psi_i}\jump{\psi_i} \, ds > 0.
\end{equation}
On the other hand, for $i \neq j$, $
\tenq{A}_{22}(i,j) =  \sum_{e \in \mathcal{E}_h} \frac{1}{|e|} \int_e \sigma(\beta) \jump{\psi_j}\jump{\psi_i}\, ds < 0.$
Furthermore, we have $\sum_{j=1, j\neq i}^{N_e} \left |A_{ij} \right | \leq A_{ii}, \,\,\forall i \geq 1$ and the strict inequality can be achieved when the triangle $T_i \in \mathcal{T}_h$ has at least one edge that is not shared by a neighbor triangle. This shows that the matrix $\tenq{A}_{22}$ is weakly diagonally dominant. This completes the proof.
\end{proof}

\section{Numerical experiments}

In this section, we present some numerical experiments to confirm the theoretical developments. We shall report the numerical error estimate of EIFEM as well as the efficiency of auxiliary space preconditioner.

We consider a model equation on $\Omega = (-1,1)^2$ with a circle-shaped interface $x^2+y^2=0.4^2$. The analytic solution is well-known to be given as
\begin{eqnarray*}
\displaystyle
p = \left\{
	\begin{array}{ll}
	r^3/\beta^-   & \text{in } \Omega^-, \\
	r^3/\beta^+ + \left(\frac{1}{\beta^-} - \frac{1}{\beta^+} \right)0.4^3 & \text{in } \Omega^+,
	\end{array}
	\right.
\end{eqnarray*}
where a number of jump discontinuities of $\beta$, i.e., $\beta^+$ and $\beta^-$ have been attempted. Note that the numerical solutions were conducted on a uniform triangulation $\T_h$ by rectangles whose size is $h$. We report the results with the various contrast of $\beta^+$ and $\beta^-$ across the interface, i.e., $\beta^-/\beta^+=1,10,100,1000$. The graphs of numerical solution of primary and velocity variable for the case of $(\beta^-, \beta^+)=(100,1)$ are reported in Figure \ref{fig:interel}.

\subsection{Numerical error analysis of EIFEM}

We report the $L^2$ and $H^1$-errors of the primary variable $p$, in Table 1, 3, 5, 7, and $L^2$, $H({\rm div})$ and local conservation errors of the velocity variable in Table 2, 4, 6, 8. Note that the local conservation error is defined by
\begin{equation*}
\| \nabla \cdot \bu_h - f \|_{L^\infty(\mathcal{T}_h)} := \max_{T\in\mathcal{T}_h} \norm{\nabla \cdot \bu_h - f }_{L^\infty(T)}.
\end{equation*}
The coefficients contrast are: $\beta^-/\beta^+=1$ for Table 1 - 2, $\beta^-/\beta^+=10$ for Table 3-4, $\beta^-/\beta^+=100$ for Table 5-6 and $\beta^-/\beta^+=1000$ for Table 7-8. In all cases, we observe the optimal convergency for both the variables in terms of $L^2$, $H^1$ and $H({\rm div})$ for different ratio of the contrast in $\beta$ as predicted by the theory. The local conservation errors are also observed to be below $E-11$ when $h < 1/64$, showing that our scheme is indeed locally conservative.

\begin{figure}[ht]
  \begin{center}
 \includegraphics[width=5.5cm]{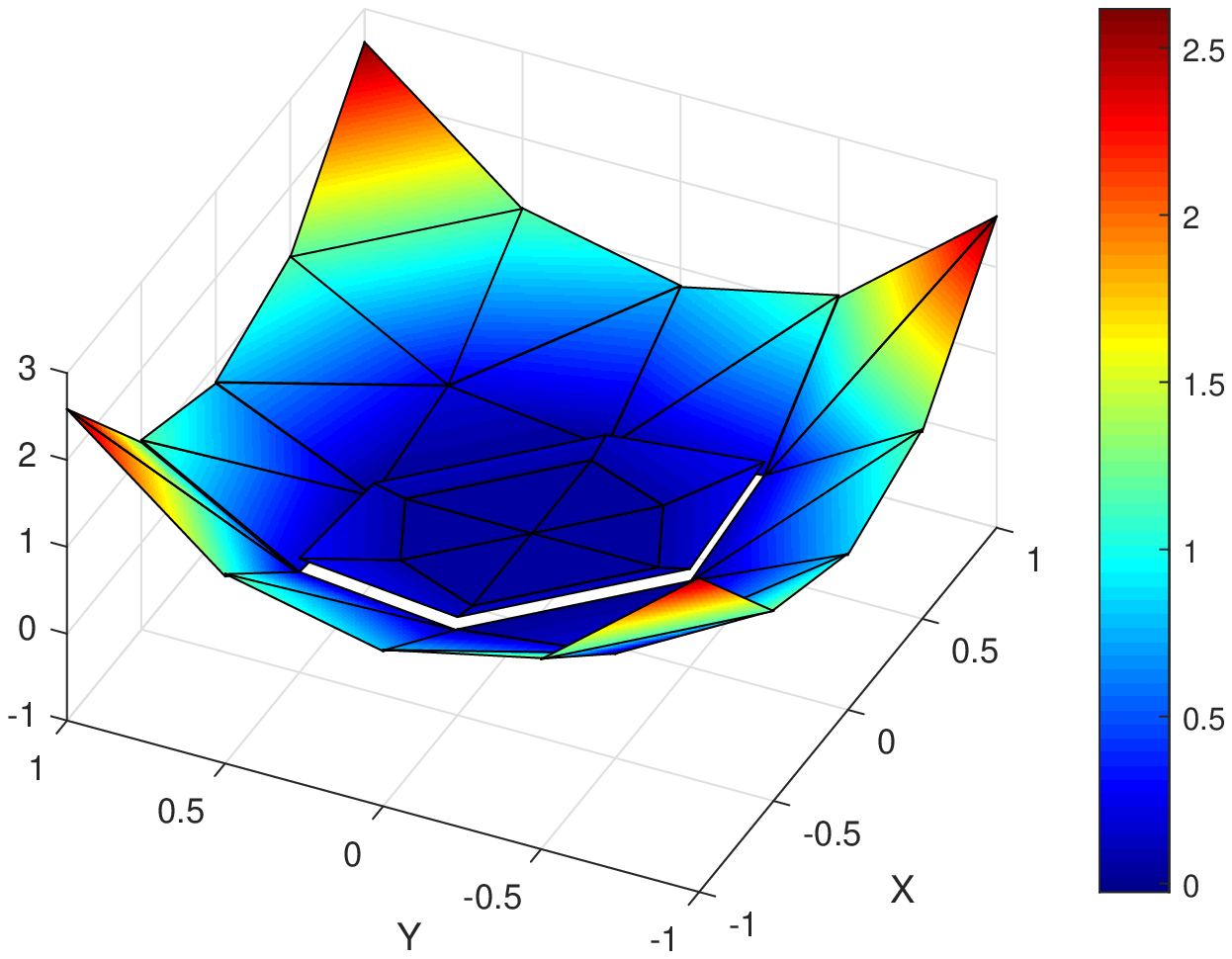}
 \includegraphics[width=5.5cm]{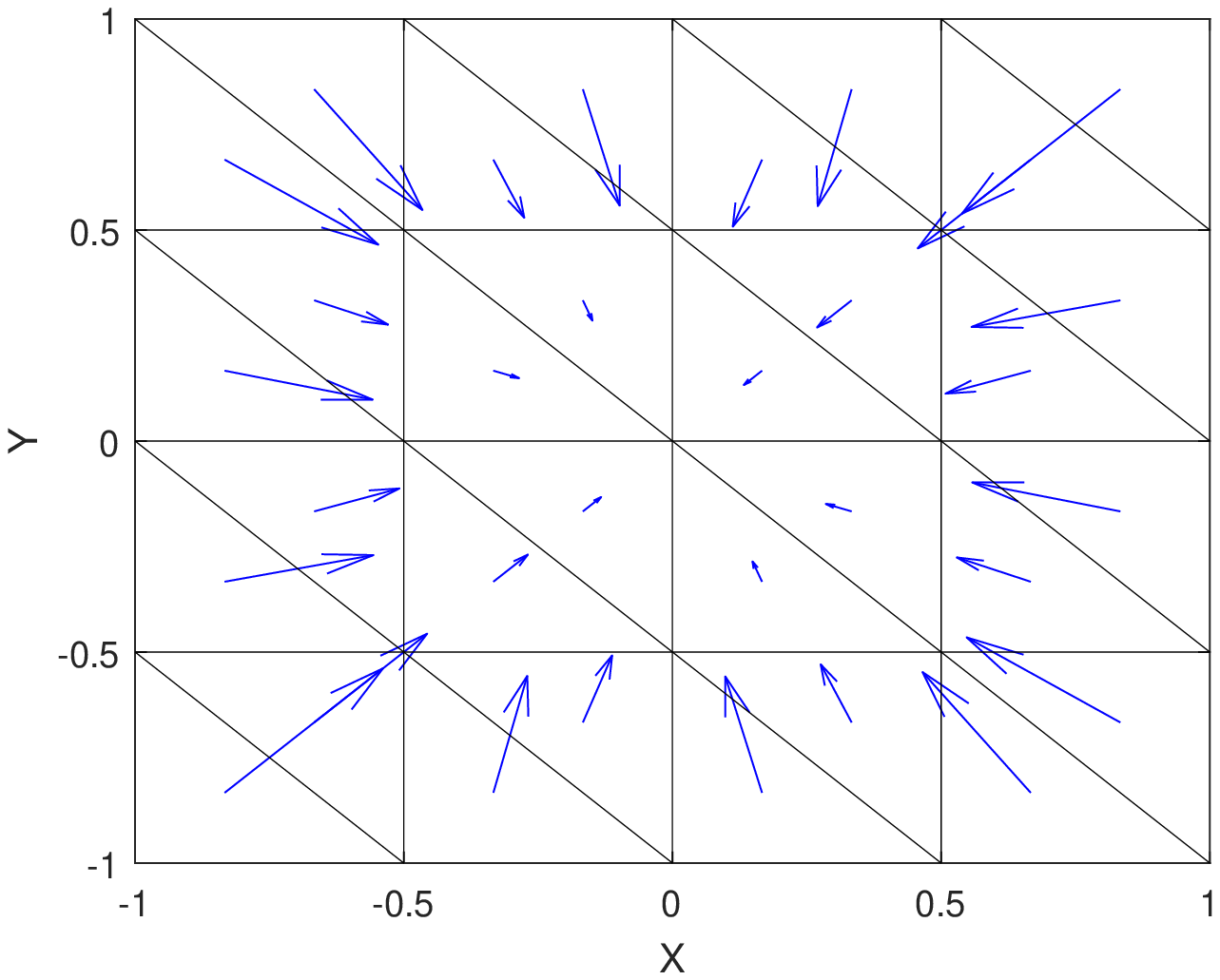}
  \caption{Numerical solution for pressure and velocity for $(\beta^-,\beta^+) = (100,1)$.}
 \label{fig:interel}
\end{center}
\end{figure}

\begin{table}
\begin{center}
\begin{tabular}{||c|c|c||c|c||}
\hline
$1/h$  & $\norm{p-p_{h}}_{0,\Omega}$  & Order & $\norm{p-p_{h}}_{1,h}$ & Order  \\ 
\hline \hline
$16$ & 2.242 E-3                                    &      x   & 2.044 E-1                      &  x \\
\hline
$32$ & 5.850 E-4  & 1.938 & 1.021 E-1  & 1.002 \\
\hline
$64$ & 1.493 E-4  & 1.970 & 5.102 E-2  & 1.001 \\
\hline
$128$& 3.773 E-5  & 1.985 & 2.550 E-2  & 1.001 \\
\hline
$256$& 9.480 E-6  & 1.993 & 1.275 E-2  & 1.000 \\
\hline
$512$& 2.376 E-6  & 1.996 & 6.373 E-3  & 1.000 \\
\hline
\end{tabular}
\end{center}\caption{Potential Error Behavior in $L^2$ and Energy norm for $(\beta^-,\beta^+) = (1,1)$.}
\end{table}

\begin{table}
\begin{center}
\begin{tabular}{||c|c|c||c|c|c||}
  \hline
  $1/h$  & $\norm{\bu-\bu_{h}}_{0,\Omega}$  & Order & $\norm{\nabla \cdot (\bu-\bu_{h})}_{0,\Omega}$ & Order & $\| \nabla \cdot \bu_h - f \|_{L^\infty(\mathcal{T}_h)}$ \\
  \hline \hline
  $16$ & 7.360 E-2 &   x      & 2.651 E-1 & x & 4.244 E-08 \\
  \hline
  $32$ & 3.655 E-2 & 1.010 & 1.326 E-1 & 1.000 & 1.675 E-09 \\
  \hline
  $64$ &  1.823 E-2 & 1.004 & 6.629 E-2 & 1.000 & 6.179 E-11 \\
  \hline
  $128$& 9.106 E-3 & 1.002 & 3.315 E-2 & 1.000 & 3.768 E-12 \\
  \hline
  $256$& 4.551 E-3 & 1.001 & 1.657 E-2 & 1.000 & 5.439 E-12 \\
  \hline
  $512$& 2.275 E-3 & 1.000 & 8.269 E-3 & 1.000 & 3.727 E-12 \\
  \hline
  \end{tabular}
\end{center}
\caption{Flux Error Behavior in $L^2$, $H({\rm div},\Omega)$ and Local conservation for $(\beta^-,\beta^+) = (1,1)$.}
 \label{Error_Table1}
\end{table}

\begin{table}
\begin{center}
\begin{tabular}{||c|c|c||c|c||}
\hline
$1/h$  & $\norm{p-p_{h}}_{0,\Omega}$  & Order & $\norm{p-p_{h}}_{1,h}$ & Order  \\ 
  \hline
  $16$ & 2.381 E-3  &     x  & 2.029 E-1 &  x \\ \hline
  $32$ & 6.174 E-4  & 1.947 & 1.013 E-1 & 1.002 \\ \hline
  $64$ & 1.581 E-4  & 1.966 & 5.063 E-2 & 1.001 \\ \hline
  $128$& 3.996 E-5  & 1.984 & 2.531 E-2 & 1.000 \\ \hline
  $256$& 1.005 E-5  & 1.992 & 1.265 E-2 & 1.000 \\  \hline
  $512$& 2.516 E-6  & 1.998 & 6.325 E-3 & 1.000 \\  \hline
  \end{tabular}
\end{center} \caption{Potential Error Behavior in $L^2$ and Energy norm for $(\beta^-,\beta^+) = (10,1)$.}
\end{table}

\begin{table}
\begin{center}
\begin{tabular}{||c|c|c||c|c|c||}
  \hline
  $1/h$  & $\norm{\bu-\bu_{h}}_{L^2(\Omega)}$  & order & $\norm{\nabla \cdot (\bu-\bu_{h})}_{0,\Omega}$ & order & $\| \nabla \cdot \bu_h - f \|_{L^\infty(\mathcal{T}_h)}$ \\
  \hline
  $16$ & 9.714 E-2 &  x    & 2.651 E-1 & x & 4.244 E-08 \\ \hline
  $32$ & 6.210 E-2 & 1.003 & 1.326 E-1 & 1.000 & 1.675 E-09 \\ \hline
  $64$&  2.186 E-2 & 1.030 & 6.629 E-2 & 1.000 & 6.186 E-11 \\ \hline
  $128$& 9.591 E-3 & 1.010 & 3.315 E-2 & 1.000 & 4.131 E-12 \\ \hline
  $256$& 4.685 E-3 & 1.003 & 1.657 E-2 & 1.000 & 5.669 E-12 \\  \hline
  $512$& 2.278 E-3 & 1.001 & 8.286 E-3 & 1.000 & 3.018 E-12 \\  \hline
  \end{tabular}
\end{center} \caption{Flux Error Behavior in $L^2$, $H({\rm div},\Omega)$ and Local conservation for $(\beta^-,\beta^+) = (10,1)$.}
 \label{Error_Table2}
\end{table}

\begin{table}
\begin{center}
\begin{tabular}{||c|c|c||c|c||}
\hline
$1/h$  & $\norm{p-p_{h}}_{0,\Omega}$  & Order & $\norm{p-p_{h}}_{1,h}$ & Order  \\ \hline
  $16$ & 2.357 E-3  &  x    & 2.031 E-1 & x \\ \hline
  $32$ & 6.059 E-4  & 1.960 & 1.014 E-1 & 1.003 \\ \hline
  $64$ & 1.578 E-4  & 1.941 & 5.064 E-2 & 1.001 \\ \hline
  $128$& 4.021 E-5  & 1.973 & 2.531 E-2 & 1.001 \\ \hline
  $256$& 1.011 E-5  & 1.991 & 1.265 E-2 & 1.000 \\  \hline
  $512$& 2.534 E-6  & 1.997 & 6.325 E-3 & 1.000 \\  \hline
  \end{tabular}
\end{center}\caption{Potential Error Behavior in $L^2$ and Energy norm for $(\beta^-,\beta^+) = (100,1)$.}
\end{table}

\begin{table}
\begin{center}
\begin{tabular}{||c|c|c||c|c|c||}
  \hline
  $1/h$  & $\norm{\bu-\bu_{h}}_{L^2(\Omega)}$  & order & $\norm{\nabla \cdot (\bu-\bu_{h})}_{0,\Omega}$ & order & $\| \nabla \cdot \bu_h - f \|_{L^\infty(\mathcal{T}_h)}$ \\ \hline
  $16$ & 9.714 E-2 &    x  & 2.651 E-1 & x & 4.244 E-08     \\ \hline
  $32$ & 6.210 E-2 & 0.646 & 1.326 E-1 & 1.000 & 1.675 E-09 \\ \hline
  $64$&  2.186 E-2 & 1.506 & 6.629 E-2 & 1.000 & 6.459 E-11 \\ \hline
  $128$& 9.591 E-3 & 1.189 & 3.315 E-2 & 1.000 & 4.409 E-12 \\ \hline
  $256$& 4.685 E-3 & 1.034 & 1.657 E-2 & 1.000 & 1.587 E-11 \\  \hline
  $512$& 2.358 E-3 & 0.991 & 8.286 E-3 & 1.000 & 4.989 E-12 \\  \hline
  \end{tabular}
\end{center} \caption{Flux Error Behavior in $L^2$, $H({\rm div},\Omega)$ and Local conservation for $(\beta^-,\beta^+) = (100,1)$.}
\label{Error_Table3}
\end{table}

\begin{table}
\begin{center}
\begin{tabular}{||c|c|c||c|c||}
\hline
  $1/h$  & $\norm{p-p_{h}}_{L^2(\Omega)}$  & Order & $\norm{p-p_{h}}_{1,h}$ & Order  \\
  \hline
  $16$ & 2.372 E-3  &  x     & 2.037 E-1  &  x \\ \hline
  $32$ & 6.283 E-4  & 1.917 & 1.017 E-1 & 1.002 \\ \hline
  $64$ & 1.569 E-4  & 2.002 & 5.069 E-2 & 1.005 \\ \hline
  $128$& 3.991 E-5  & 1.974 & 2.531 E-2 & 1.002 \\ \hline
  $256$& 1.009 E-5  & 1.984 & 1.265 E-2 & 1.000 \\  \hline
  $512$& 2.531 E-6  & 1.995 & 6.326 E-3 & 1.000 \\  \hline
  \end{tabular}
\end{center}\caption{Potential Error Behavior in $L^2$ and Energy norm for $(\beta^-,\beta^+) = (1000,1)$.}
\end{table}

\begin{table}
\begin{center}
\begin{tabular}{||c|c|c||c|c|c||}
  \hline
  $1/h$  & $\norm{\bu-\bu_{h}}_{L^2(\Omega)}$  & Order & $\norm{\nabla \cdot (\bu-\bu_{h})}_{0,\Omega}$ & Order & $\| \nabla \cdot \bu_h - f \|_{L^\infty(\mathcal{T}_h)}$ \\   \hline
  $16$ & 7.338 E-1 &      x   & 2.651 E-1 & x & 4.244 E-08  \\ \hline
  $32$ & 2.401 E-1 & 1.612 & 1.326 E-1 & 1.000 & 1.675 E-09 \\ \hline
  $64$&  7.518 E-2 & 1.675 & 6.629 E-2 & 1.000 & 6.325 E-11 \\ \hline
  $128$& 2.547 E-2 & 1.562 & 3.315 E-2 & 1.000 & 1.068 E-11 \\ \hline
  $256$& 8.785 E-3 & 1.535 & 1.657 E-2 & 1.000 & 6.325 E-11 \\  \hline
  $512$& 2.358 E-3 & 1.897 & 8.286 E-3 & 1.000 & 5.825 E-12\\  \hline
  \end{tabular}
\end{center}
\caption{Flux Error Behavior in $L^2$, $H({\rm div},\Omega)$ and Local conservation for $(\beta^-,\beta^+) = (1000,1)$.}
 \label{Error_Table4}
\end{table}

\subsection{Performance of auxiliary space preconditioner for EIFEM}
In this subsection, we demonstrate the performance of auxiliary space preconditioner for EIFEM introduced and analyzed in Section 5. Note that AMG is used as a preconditioner for $\tenq{A}_{11}$ and $\tenq{A}_{22}$. For the implementation of AMG preconditioner, we employed the c$++$ open software library AMGCL developed by Demidov \cite{demidovamgcl,demidov2019amgcl}. All experiments were conducted on PC with an Intel(R) Core(TM) i7-3770 CPU @ 3.40GHz processor.
Stopping criteria were given as a relative residual less than $10^{-7}$. In each auxiliary space preconditioner, maximum number of AMG iterations for $\tenq{A}_{11}$ and $\tenq{A}_{22}$ was set at the fixed number $5$. Namely, we use here the approximate $\tenq{A}_{11}$ and $\tenq{A}_{22}$. 

The PCG iteration number and total CPU time are reported in Table 9-10. We use only one Gauss-Seidel smoothing as a part of the auxiliary space preconditioning. The corresponding result is reported in Table 9. Note that the case when no smoothing is used is also attempted and this is reported in Table 10. Our observation is that with the addition of smoothing, the auxiliary space solver works as robust solver as theory predicted. The PCG iteration numbers are bounded as $h$ gets smaller.
Thus, the computational costs for PCG preconditioned by Algorithm 5.1. is justified to be of $\mathcal{O}(N)$, where $N$ is the number of unknowns. We would like to note that Algorithm 5.1. with no smoothing works fine for some special case when the contrast ratio of $\beta$ is not too large as reported in Table 10.

\begin{table}[!ht]
\centering
\begin{tabular}{r|rr}
\hline
Case 1. & PCG it &  CPU time  \\
$1/h$	&  &  \\
\hline
$32$ & 11  & 0.447 \\
$64$ & 11  & 0.803 \\
$128$ & 11 & 3.235 \\
$256$ & 11 & 14.587 \\
$512$ & 11 & 46.262 \\
\hline
\end{tabular}
\begin{tabular}{r|rr}
\hline
Case 2. & PCG it &  CPU time  \\
$1/h$	&  &  \\
\hline
$32$ & 11  & 0.529 \\
$64$ & 11  & 0.912 \\
$128$ & 11 & 2.893 \\
$256$ & 11 & 14.115 \\
$512$ & 11 & 51.709 \\
\hline
\end{tabular}
\begin{tabular}{r|rr}
\hline
Case 3. & PCG it &  CPU time  \\
$1/h$	&  &  \\
\hline
$32$ & 12  & 0.43 \\
$64$ & 13  & 1.16 \\
$128$ & 11 & 3.202 \\
$256$ & 11 & 11.956 \\
$512$ & 11 & 53.461 \\
\hline
\end{tabular}
\begin{tabular}{r|rr}
\hline
Case 4. & PCG it &  CPU time  \\
$1/h$	&  &  \\
\hline
$32$ & 14 & 0.52 \\
$64$ & 18 & 1.481 \\
$128$ & 20 & 5.615 \\
$256$ & 22 & 23.983 \\
$512$ & 21 & 96.468 \\
\hline
\end{tabular}
\caption{Performance of auxiliary space preconditioner-PCG with nGS=1.
Case 1., Case 2., Case 3. and  Case 4. correspond to $(\beta^-,\beta^+)=(1,1)$, $(\beta^-,\beta^+)=(1,10)$, $(\beta^-,\beta^+)=(1,100)$ and $(\beta^-,\beta^+)=(1,1000)$ respectively.}
\label{tab:CPU_TIMES_NGS1}
\end{table}

\begin{table}[!ht]
\centering
\begin{tabular}{r|rr}
\hline
Case 1. & PCG it &  CPU time  \\
$1/h$	&  &  \\
\hline
$32$ & 11  & 0.375 \\
$64$ & 11  & 0.742 \\
$128$ & 11 & 2.55 \\
$256$ & 11 & 10.094 \\
$512$ & 11 & 49.019 \\
\hline
\end{tabular}
\begin{tabular}{r|rr}
\hline
Case 2. & PCG it &  CPU time  \\
$1/h$	&  &  \\
\hline
$32$ & 13  & 0.505 \\
$64$ & 13  & 0.801 \\
$128$ & 13 & 2.646  \\
$256$ & 13 & 10.994 \\
$512$ & 13 & 72.935 \\
\hline
\end{tabular}
\begin{tabular}{r|rr}
\hline
Case 3. & PCG it &  CPU time  \\
$1/h$	&  &  \\
\hline
$32$ & 12  & 0.587 \\
$64$ & 13  & 1.134 \\
$128$ & 11 & 3.688 \\
$256$ & 11 & 15.95 \\
$512$ & 18 & 86.001  \\
\hline
\end{tabular}
\begin{tabular}{r|rr}
\hline
Case 4. & PCG it &  CPU time  \\
$1/h$	&  &  \\
\hline
$32$ & 23 & 0.906 \\
$64$ & 43 & 2.889 \\
$128$ & 83 & 18.6089 \\
$256$ & 51 & 41.001 \\
$512$ & 29 & 124.116 \\
\hline
\end{tabular}
\caption{Performance of auxiliary space preconditioner - PCG with nGS=0.
Case 1., Case 2., Case 3. and  Case 4. correspond to $(\beta^-,\beta^+)=(1,1)$, $(\beta^-,\beta^+)=(1,10)$, $(\beta^-,\beta^+)=(1,100)$ and $(\beta^-,\beta^+)=(1,1000)$ respectively.}
\label{tab:CPU_TIMES_NGS0}
\end{table}

\section{Concluding remarks}

In this paper, we have developed the locally conservative immersed finite element as well as a fast solver based on auxiliary space preconditioning. In our future work, we shall extend this method for coupled flow and transports as well as elasticity with interface.

\section{Declarations} 

The data that support the findings of this study are available from the corresponding author upon reasonable request.


\begin{acknowledgements}

First and second author are supported by the National Research Foundation of Korea (NRF) grant funded by the Korea government (MSIT) 
(No. 2020R1C1C1A01005396).
Third author is supported by Brain Pool Program through the National Research Foundation of Korea(NRF) funded by the Ministry of Science 
and ICT (grant number) (NRF-2020H1D3A2A01041079)
\end{acknowledgements}
 
%
%



\bibliographystyle{siam}
\bibliography{ref}

%
%

\end{document}